\documentclass[11pt]{article}
\usepackage{latexsym, amsmath, amsfonts, amscd, amssymb, verbatim}

\usepackage{color}
\usepackage{xcolor}
\usepackage{graphicx}
\normalsize

\newcommand{\N}{I\!\!N}

\newcommand{\dx}{\,\mathrm{d}x}
\newcommand{\dt}{\mathrm{d}t}

\def\A{{\mathcal A}}

 \def\al{{\alpha}} 
 
\newcommand{\q}{\quad} 
\newcommand{\qq}{\qquad}
\newtheorem{Theorem}{Theorem}[section]
\newtheorem{Proposition}{Proposition}[section]
\newtheorem{Remark}{Remark}[section]

\newtheorem{Definition}{Definition}[section]
 
\newtheorem{Assumption}{Assumption}[section]

\newenvironment{proof}[1]{%
\par\vspace{1\baselineskip}%
\noindent{\bf Proof#1.\  }\ignorespaces }{%
\nobreak\hfill\mbox{\ \ $\Box$}%
\par\vspace{1\baselineskip}}%

\newcommand{\bd}{\begin{displaymath}}
\newcommand{\ed}{\end{displaymath}}
\newcommand{\be}{\begin{equation}}
\newcommand{\ee}{\end{equation}}
\newcommand{\bea}{\begin{eqnarray}}
\newcommand{\eea}{\end{eqnarray}}
\newcommand{\bda}{\begin{eqnarray*}}
\newcommand{\eda}{\end{eqnarray*}}
\newcommand{\ba}{\begin{array}}
\newcommand{\ea}{\end{array}}

\newcommand{\B}{I\kern -.35em B}
\newcommand{\R}{I\kern -.35em R}

\newcommand{\Mt}{\Rightarrow}
\newcommand{\To}{\rightrightarrows}
\newcommand{\mt}{\mapsto}
\newcommand{\st}{\subset}

\newcommand{\Lra}{\Longrightarrow}

\newcommand{\U}{{\mathcal U}}

\newcommand{\e}{\varepsilon}
\newcommand{\ph}{\varphi}
\newcommand{\sth}{ \, :\;}

\newcommand{\dd}{\mbox{\rm\,d}}

\renewcommand{\lll}{\langle}
\newcommand{\rrr}{\rangle}

\newcommand{\X}{{\mathcal X}}
\newcommand{\Y}{{\mathcal Y}}

\newcommand{\F}{{\mathcal F}}
\renewcommand{\O}{{\mathcal O}}

\def\R{\mathbb{R}}
\newcommand{\reff}{\eqref}
\newcommand{\bino}{\bigskip\noindent}

\newcommand{\Dex}{{\Delta x}}

\definecolor{dgreen}{rgb}{0,0.8,0}

\begin{document}

\title{Strong metric (sub)regularity in optimal control%
\thanks{This study is financed by the European Union-NextGenerationEU, through the
National Recovery and Resilience Plan of the Republic of Bulgaria, project
№ BG-RRP-2.004-0008-C01. The last author is also supported by the Austrian Science Foundation (FWF) under grant 
            No I-4571-N. }}
\author{
N.A. Jork\thanks{University Tübingen, Germany, {\tt nicolai.jork@uni-tuebingen.de}} \and
N.P. Osmolovskii\thanks{Systems Research Institute, Polish Academy of Sciences,  
Warsaw, Poland  {\tt osmolov@ibspan.waw.pl}} 
 \and  V.M. Veliov\thanks{Institute of Statistics and Mathematical Methods in Economics,
TU Wien, Austria, {\tt vladimir.veliov@tuwien.ac.at}} }

\date{}

\maketitle

\vspace{-1cm}

\begin{center} 
{\em Dedicated to the honor of Terry Rockafellar}
\end{center}

\begin{abstract}

This is mainly a survey on the properties of Strong Metric Regularity (SMR)
and Strong Metric subRegularity (SMsR) 
of mappings representing first order optimality conditions (so-called {\em optimality mappings})
of optimization problems in infinite dimensional spaces. The focus is on the optimality mappings
associated with optimal control problems for ODE systems or PDEs. We especially emphasize 
an extension of the concepts of SMR and SMsR which involves two metrics either in the domain
or in the image spaces. The paper shows the relevance of this extension in optimal
control.   
\end{abstract}



\section{Introduction} \label{SIntro}

In this paper, we present a survey about recent results on two properties of
strong metric regularity of mappings, focusing on the so-called optimality mappings 
in calculus of variations and optimal control. These regularity properties imply useful stability 
properties of the solutions of such problems and estimates for the convergence rate of various kinds of 
numerical approximations.   

We begin with two definitions.
Let $(\X,d_\X)$ and $(\Y,d_\Y)$ be metric spaces. Let $d^\circ_{\X}$ and $d^\circ_{\Y}$
be additional metrics in the sets $\X$ and $\Y$, respectively. The symbols 
$\O_\X(x)$ and $\O^\circ_\X(x)$ denote open neighborhoods of $x \in \X$ with respect to 
$d_\X$ and $d_\X^\circ$, correspondingly. 
The meaning of $\O_\Y(y)$ and $\O^\circ_\Y(y)$ is similar for $y \in \Y$. Given a set-valued mapping 
$\F:\X\To\Y$, the inverse mapping, $\F^{-1}: \Y \To \X$, is the set-valued mapping 
defined as $\F^{-1}(y):=\{ x \in \X: y \in \F(x)\}$.

\begin{Definition} \label{DSMHsR}
The set-valued mapping $\F:\X \To \Y$ is {\em Strongly Metrically H\"older sub-Regular} 
(SMHsR) at $(\hat x, \hat y) \in \X \times \Y$ 
(with metrics $d_\X, d_\X^\circ$ in $\X$ and $d_\Y, d_\Y^\circ$ in $\Y$) if 
$\hat y \in \F(\hat x)$ and there exist numbers $\kappa, \, \beta > 0$ and
neighborhoods $\O_\X(\hat x)$ and $\O_\Y(\hat y)$ such that 
\be \label{ESMHsR}
         \forall (x,y) \in \O_\X(\hat x) \times \O_\Y(\hat y) \; \mbox{ with } \; y \in \F(x)
         \;\;\;\Lra  \;\;\; d^\circ_\X(x,\hat x) \leq  \kappa \,d^\circ_\Y(y, \hat y)^\beta.
\ee
\end{Definition}

\bino
\begin{Definition} \label{DSMHR}
The set-valued mapping $\F:\X \To \Y$ is {\em Strongly Metrically H\"older Regular} 
(SMHR) around $(\hat x, \hat y) \in \X \times \Y$ 
(with metrics $d_\X, d_\X^\circ$ in $\X$ and $d_\Y, d_\Y^\circ$ in $\Y$)
if $\hat y \in \F(\hat x)$ and
there exist numbers $\kappa, \, \beta > 0$ and
neighborhoods $\O_\X(\hat x)$ and $\O_\Y(\hat y)$ such that 
the mapping $\F^{-1}(\cdot) \cap \O_\X(\hat x)$ is non-empty valued on $\O_\Y(\hat y)$ and 
\be \label{ESMHR}
         \forall (x,y), (x',y') \in \O_\X(\hat x) \times \O_\Y(\hat y) \; 
         \mbox{ with } \; y \in \F(x), \;\; y' \in \F(x') 
         \;\;\;\Lra  \;\;\; d^\circ_\X(x,x') \leq  \kappa \,d^\circ_\Y(y, y')^\beta.
\ee
%
\end{Definition}

Obviously (by taking $y = \hat y$ in the definition), SMHsR implies that 
$\F^{-1}(\hat y)$ is a singleton, while $\F^{-1}(y)$ may be empty for $y \not= \hat y$. 
Similarly, if $F$ is SMHR then $\F^{-1}(\cdot) \cap \O_\X(\hat x)$
is a single-valued H\"older continuous function in $\O_\Y(\hat y)$ 
(Lipschitz continuous if $\beta = 1$).

If some of the regularity properties defined above are fulfilled with $\beta = 1$ 
we skip the name  ``H\"older'' and shorten the acronyms as SMsR and SMR. 

The elements $y \in \Y$ can be considered as ``disturbances'' of the reference point $\hat y$, 
so that it is all about stability of the solution $\hat x$ of $\hat y \in \F(x)$ with respect to (``small'') 
disturbances $y$.
In case of linear spaces $\X$ and $\Y$, consideration of two metrics in the domain space $\X$
is often aimed to ensure differentiability of the functions involved in $\F$ in the stronger norm,
presumably $d_\X \geq d^\circ_\X$, while the stronger norm in the image space $\Y$, 
presumably $d_Y$, ensures appropriate properties of the disturbances $y \in \Y$ for which
\reff{ESMHsR}, resp. \reff{ESMHR}, is fulfilled.  


\bino
Below we make some remarks concerning the history of the strong metric regularity properties
defined above (we do not discuss in this paper the non-strong versions of these
properties\footnote{See e.g. \cite[Chapter 3]{AD+TR-book2} for the notions of Metric Regularity and 
Metric Subregularity. For important classes of mappings between finite dimensional
spaces this notions actually coincide: KKT mappings \cite[Theorem 4I.2]{AD+TR-book2}, 
maximal monotone mappings \cite{Rock-23}, variational inequalities over convex polyhedral 
sets \cite{Dont+Rock-96}.}), starting with the Lipschitz case $\beta =1$, and 
with a single metric in each space: $d_\X = d_\X^\circ$ and $d_\Y = d_\Y^\circ$.
The name {\em strong regularity} was introduced by Robinson \cite{Rob-80}, while
the name {\em metric regularity} was coined by Borwein \cite{Borwein-86}.  
The property of SMsR was introduced by Dontchev \cite{Dont-95} and named in this way
in \cite{AD+TR-04}; see also \cite[Chapter~3.9]{AD+TR-book2} 
and the recent paper \cite{Cibulka+Dontchev+Kruger-18}.
However, versions of these properties have been used under several other names
by Bonnans, Shapiro, Klatte, Kummer and others, 
\cite{Bonnans-94,Bonn+Shap-book-2000,Klatte+Kummer-2002-book}. 

The H\"older regularity notions ($\beta < 1$), as defined above, 
were studied by Frankowska and Quincampoix \cite{Fank+Quin-12}, although the idea 
to study nonlinear estimations such as \reff{ESMHsR} or \reff{ESMHR} 
go back to works by Borwein and Zhuang (1988), Penot (1989), 
Frankowska (1987, 1990).
The general topic of {\em nonlinear regularity} is reviewed in detail by Ioffe \cite{Ioffe-17}.

Two norms in the image space $\Y$ were first used by Quincampoix and Veliov in
\cite{MQ+VV:SICON-13}, where the property of {\em strong bi-metric regularity} 
was introduced. Utilization of two norms also in the domain space $\X$ 
proved to be useful; further quotations will be given in the corresponding sections. 
We mention that the two regularity properties with four norms defined above can be considered 
as a special case of the {\em regularity on a fixed set}, which is a non-local property;
see Ioffe \cite{Ioffe:11_fixed_set}. However, the analysis in the present paper is
local and we deal with sets that are not necessarily fixed -- these are the neighborhoods 
$\O_\X(\hat x)$ and $\O_\Y(\hat y)$ in Definitions \ref{DSMHsR} and \ref{DSMHR}, 
which may be taken sufficiently small.  

The present survey is far from being comprehensive: a lot is missing concerning the general
studies of regularity of mappings or, specifically, of mappings associated with variational inequalities.
The paper does not include the numerous results about stability of solutions of optimization 
problems with respect to perturbations, for which the regularity theory of mappings is 
an important tool. Only a little is mentioned about applications of ideas and results
in the regularity theory to numerical approximations. We refer to the books by
Bonnans and Shapiro \cite{Bonn+Shap-book-2000}, 
Dontchev and Rockafellar \cite{AD+TR-book2},
Mordukhovich \cite{Morukh_book12-06},
Klatte and Kummer \cite{Klatte+Kummer-13},
Ioffe \cite{Ioffe-17}
for further information.
Instead, in this paper we focus on regularity properties of the optimality mapping associated
with optimal control problems, especially in cases when two norms 
in either the domain or in the image space naturally appear in metric regularity issues.
 
\bino
In the next section we present some basic abstract results about stability of the two regularity 
defined above.
In Section \ref{SMP} we present sufficient conditions for the SMsR property of a mapping 
defined by the Karush-Kuhn-Tucker system for mathematical programming problems 
in Banach spaces. 
Although this result concerns infinite-dimensional problems and
is applicable to problems of calculus of variations, it does not apply to problems of 
optimal control, where the control constraint cannot be represented by a finite 
number of inequalities. Results about SMsR and SMR of the optimality mapping 
associated with various optimal control problems are surveyed in the subsequent three 
sections: coercive Mayer-type optimal control problems for ODE systems, affine optimal control
problems for ODE systems, an optimal control problem for a semi-linear parabolic equations.
In all cases, we focus on results where consideration of two norms either in $\X$ or in $\Y$
is essential. In addition, we give a broader literature review of related results.

\section{Preliminaries} \label{SPrelim}

Below we present two theorems of Robinson's kind, claiming invariance of the metric regularity
properties under appropriately ``small'' perturbations, in particular under ``linearization''.  
The notions and notations introduced in the introduction will be used.
In addition we denote by $\B_\Y(y;\gamma)$ the ball in the space $(\Y,d_\Y)$
centered at $y$ and with radius $\gamma$. 

\begin{Theorem} \label{T_SMsR-Rob}
Assume that $\Y$ is a linear space and the metrics $d_\Y$ and $d^\circ_\Y$ are shift-invariant.
Assume also that $\F :\X \Mt \Y$ is SMsR (with metrics $d_\X, d_\X^\circ$ in $\X$ and 
$d_\Y, d_\Y^\circ$ in $\Y$) at $(\hat x, \hat y)$ with neighborhoods 
$\O_\X(\hat x)$, $\O_\Y(\hat y)$ and a constant $\kappa > 0$ (see Definition \ref{DSMHsR}
with $\beta = 1$).
Let $\O'_\X(\hat x)$ be a neighborhood of $\hat x$ and $\ph: \O'_\X(\hat x) \to \Y$. 
Let the numbers $\mu$, $\gamma$, and the neighborhood 
$\O'_\Y(\hat y)$ of $\hat y$ satisfy the relations
\bda
            && \O'_\X(\hat x) \st \O_\X(\hat x) , \qquad 
               \O'_\Y(\hat y) + \B_\Y(0;\gamma) \st \O_\Y(\hat y), 
               \qquad \mu \kappa < 1, \\
         && d_\Y(\ph(x),\ph(\hat x)) \leq \gamma, \quad 
     d^\circ_\Y(\ph(x),\ph(\hat x)) \leq \mu \,d^\circ_\X(x,\hat x)  \qquad \forall \, x \in \O_\X(\hat x).
\eda
Then the mapping $\F + \ph$ is SMsR at $(\hat x, \hat y + \ph(\hat x))$ with 
neighborhoods $\O'_\X(\hat x)$ and $\O'_\Y(\hat y) + \ph(\hat x)$ and a constant
$\kappa' := \kappa/(1-\mu \kappa)$.
\end{Theorem}

\begin{proof}{} The proof adapts that of Theorem 2.1 in \cite{Cibulka+Dontchev+Kruger-18}.
Let $x \in \O'_\X(\hat x)$, $y \in \O'_\Y(\hat y) + \ph(\hat x)$ and
$y \in \F(x) + \ph(x)$. Then  $x \in \O_\X(\hat x)$ and
\bd
        \F(x) \ni y - \ph(x) \in \O'_\Y(\hat y) + \ph(\hat x) - \ph(x)  \st 
        \O'_\Y(\hat y) + \B_\Y(0; \gamma) \st  \O_\Y(\hat y).
\ed
Hence,
\bda
         d_\X^\circ(x,\hat x) &\leq& \kappa d_\Y^\circ(y-\ph(x),\hat y) 
         \leq \kappa [d^\circ_\Y(y - \ph(\hat x), \hat y) + d^\circ_\Y(\ph(\hat x), \ph(x))] \\
         &\leq& \kappa  d^\circ_\Y(y - \ph(\hat x), \hat y) + \mu \kappa  d_\X^\circ(x,\hat x),
\eda
which implies the SMsR property of $\F+\ph$ at $(\hat x, \hat y + \ph(\hat x))$
with the specified neighborhoods and constant $\kappa'$.
\end{proof}

Notice that in the case $\ph(\hat x) = 0$ the mapping $\F + \ph$ is SMsR at the same point
$(\hat x, \hat y)$ as $\F$ under the assumptions that the values of $\ph$ are sufficiently small
in $d_\Y$ and $\ph$ is calm at $\hat x$ with respect to $d^\circ_\X$ and $d^\circ_\Y$ 
with a sufficiently small constant $\mu$. 
This is exactly the case in the Banach space setting if an additive single-valued 
and Fr\'echet differentiable component of $\F$ is linearized and $\ph$ is the difference between
this component and its linearization.

\begin{Theorem} \label{T_SMR-Rob}
Assume that the metric space $(\X,d_\X)$ is complete, and $d_X = d^\circ_\X$. 
Also assume that $\Y$ is a linear space, 
the metrics $d_\Y$ and $d^\circ_\Y$ are shift-invariant, 
and $d^\circ_\Y \leq d_\Y$.
Let $\F :\X \Mt \Y$ be SMR (with metrics $d_\X, d_\X^\circ$ in $\X$ and 
$d_\Y, d_\Y^\circ$ in $\Y$) around $(\hat x, \hat y)$ with neighborhoods 
$\O_\X(\hat x) = \B_\X(\hat x;a)$, $\O_\Y(\hat y) = \B_\Y(\hat y;b)$ 
and a constant $\kappa > 0$ (see Definition \ref{DSMHsR} with $\beta = 1$).
In addition, let $\ph: B_\X(\hat x;a') \to \Y$, where $a' > 0$.
Finally, let the numbers $b' > 0$, $\mu$ and $\gamma$ satisfy the relations
\bda
            && 0 < a' \leq a, \qquad  b' + \gamma \leq b, \qquad \mu \kappa < 1, 
            \qquad \kappa(b' + \gamma) \leq (1 - \kappa \mu) a', \\
         && d_\Y(\ph(x),0) \leq \gamma, \quad 
     d^\circ_\Y(\ph(x),\ph(x')) \leq \mu \,d_\X(x,x')  \qquad \forall \, x, x' \in \B_\X(\hat x;a').
\eda
Then the mapping $\F + \ph$ is SMR at $(\hat x, \hat y)$ with 
neighborhoods $B_\X(\hat x;a')$ and $\B_\Y(\hat y;b')$, and a constant
$\kappa' := \kappa/(1-\mu \kappa)$.
\end{Theorem}

\begin{proof}{} The proof is a modification of that of Theorem 4.2 in \cite{Pre_Sca_Vel_MReg-17}. 
By assumption, the mapping $y \mt s(y) :=\F^{-1}(y) \cap \B_\X(\hat x;a)$ is a Lipschitz
continuous function on $\B_\Y(\hat y;b)$ with respect to the metrics $d^\circ_\X$ and $d^\circ_\Y$.
For every $x \in \B_\X(\hat x; a') \st \B_\X(\hat x;a)$ and $y \in \B_\Y(\hat y; b')$ we have
\bd
      d_\Y(y-\ph(x), \hat y) \leq d_\Y(y - \hat y,0)+ d_\Y(\ph(x),0) \leq b' + \gamma \leq b,
\ed
thus $s(y-\ph(x))$ is defined for every such $x$ and $y$.

For an arbitrary fixed $y \in \B_\Y(\hat y;b')$ we consider the mapping 
$B_\X(\hat x;a') \ni x \mt Z_y(x) := s(y-\ph(x))$. 
We shall prove that the mapping 
$Z_y$ has a unique fixed point by using the contraction mapping theorem in the form of 
\cite[Theorem 1A.2]{AD+TR-book2}. For this we denote $\lambda = \kappa \mu < 1$ and estimate 
\bda
     d_\X(\hat x,Z_y(\hat x)) &=& 
      d_\X(s(\hat y), s(y-\ph(\hat x))) \leq \kappa \,d^\circ_\Y(\hat y,y-\ph(\hat x))\\
   &\leq& \kappa \,d_\Y(\hat y,y-\ph(\hat x)) \leq 
   \kappa (b'+ \gamma) \leq (1-\kappa \mu) a' = (1-\lambda) a'.
\eda
Moreover, for $x,\, x' \in B_\X(\hat x;a')$ we have
\bda
      d_\X(Z_y(x), Z_y(x')) &=& 
      \kappa \,d^\circ_\Y(y-\ph(x), y-\ph(x'))  = \kappa \,d^\circ_\Y(\ph(x), \ph(x')) \\
       &\leq&          \kappa \mu \,d_\X(x, x') = \lambda \,d_\X(x, x').
\eda
Then, according to \cite[Theorem 1A.2]{AD+TR-book2}, there exists a unique 
$x = x(y) \in B_\X(\hat x;a')$
such that $x = s(y - \varphi(x))$. The latter implies that $y - \varphi(x) \in \F(x)$, hence 
$x \in (\F + \ph)^{-1}(y) \cap B_\X(\hat x;a')$.
Moreover, $x(y)$ is the unique element of $(\F + \ph)^{-1}(y) \cap B_\X(\hat x;a')$. Indeed, 
if $x' \in (\F + \ph)^{-1}(y) \cap B_\X(\hat x;a')$, then $y \in \F(x') + \ph(x')$, 
hence $y - \ph(x') \in \F(x')$. Since as above we have $y-\ph(x') \in  B_\Y(\hat y;b)$ and 
$x' \in B_\X(\hat x;a')$, it also holds that $x' = s(y-\ph(x'))$. Thus $x' = x(y)$.
Thus the mapping 
$B_\Y(\hat y+\ph(\hat x);b') \ni y \mapsto (\F+\ph)^{-1}(y) \cap B_\X(\hat x,a')$ 
is single-valued.

Now, take two arbitrary elements $y, \, y' \in B_\Y(\hat y+\ph(\hat x);b')$ 
and let $x=s(y-\ph(x))$ and
$x'=s(y'-\ph(x'))$ be the unique solutions of $y\in \F(x) +\ph(x)$ in $B_\X(\hat x;a')$ 
corresponding to $y$ and $y'$, respectively. Then
\bda
   d_\X(x,x') &=& d_\X(s(y-\ph(x)),s(y'-\ph(x'))) \leq \kappa d^\circ_\Y(y-\ph(x), y'-\ph(x')) \\
   &\leq&  \kappa d^\circ_\Y(y,y') + \kappa d^\circ_\Y(\ph(x), \ph(x')) 
     \leq \kappa d^\circ_\Y(y,y') + \kappa \mu d_\X(x, x').
\eda  
Hence,
\bd
      d_\X(x,x') \leq \frac{\kappa}{1- \kappa \mu} d^\circ_\Y(y, y') 
    \leq \kappa' d^\circ_\Y(y,y'),
\ed
which completes the proof.

\end{proof}

In the proofs of the results mentioned in the next sections, 
the case $d_\X = d^\circ_X$ is actually used, while the relation $d^\circ_\Y \leq d_\Y$
with nonequivalent metrics is substantial.  

\begin{Remark} \label{RHR}{\em 
The above theorems are used to establish that a mapping of the type 
$\F + \ph$ is SMR or SMsR at (around, respectively)  $(\hat x, \hat y)$
if and only if $x \Mt \F(x) + \ph'(\hat x)(x-\hat x)$ is such, provided that 
$\ph$ is Fr\'echet differentiable. Since Theorems \ref{T_SMsR-Rob}
and \ref{T_SMR-Rob} have no counterpart in the H\"older case $\beta < 1$, 
one cannot use this linearization approach for nonlinearities in $\F + \ph$
even with a  Fr\'echet differentiable $\ph$. If the inclusion $0 \in \ph + \F$
represents first order optimality conditions for an optimization problem, 
it should be linear-quadratic in order to obtain H\"older regularity.
Such a situation will be shown in Sections \ref{SMayer} and \ref{Saffin}. 
}\end{Remark}

\bino
Characterizations of the SMsR and SMR properties of mappings between finite dimensional spaces
in terms of graphical derivatives of $\F$ are known, see Chapters 4.4 and 4.5 in    
\cite{AD+TR-book2}. Being established for finite dimensional spaces, these results 
concern only the case of a single metric in each of the spaces $\X$ and $\Y$.
In the case of Banach spaces, but also with single metrics, it is possible to
adapt characterizations of the Metric Subregularity property by coderivatives of the 
mapping $\F$ (see the survey paper \cite{Zheng-16}) to the SMsR property. 
There are no such characterization of neither the SMsR nor of the SMR in the 4-metrics case.

\section{ Mathematical programming problems in Banach spaces}\label{SMP}

Let $X$ and $Y$ be Banach spaces. Consider the problem
\bd
    \min\, \varphi(x), \q  f(x) \le 0,\q \q g(x)=0, \eqno{\rm (MP)}
\ed
where $\varphi:X \to \R$, $f=(f_1,\ldots,f_m)^*:X \to \R^m$, $g : X \to Y$ are $C^1$
mappings. The symbol $*$ denotes transposition; when applied to a space 
it denotes the dual space.

Let $\hat x\in X$ satisfy the constraints
$f(x) \le 0$, $g(x)=0$. Denote the set of active indices by
\bd
           I=\{i \sth f_i(\hat x)=0 \}.
\ed
The following assumption is called Mangasarian-Fromovitz constraint qualification (MFCQ).

\begin{Assumption}\label{assump1}    $g'(\hat x)X=Y$ and the gradients 
$f'_i(\hat x)$, $i\in I$, are  positively independent on the subspace $\ker g'(\hat x)$.
\end{Assumption} 

Recall that the  functionals $l_1,\ldots ,l_k\in X^*$ are  {\it positively independent on a 
subspace $L\subset X$} if 
\bd
      \sum_{i=1}^k \lambda _i l_i(x)=0 \q \forall\, x\in L,\q \lambda_i\ge0\q \forall\,i\q  
        \Longrightarrow \q  \lambda_i=0 \q \forall \,i.
\ed 
Let $\R^{m*}$ denote the space of row vectors of dimension $m$. 
The Karush-Kuhn-Tucker (KKT) theorem reads as follows.

\begin{Theorem} If $\hat x$ is a local minimizer in problem (MP) and Assumption \ref{assump1} holds, 
then  there exist $\hat\lambda\in \R^{m*}$ and $\hat y^*\in Y^*$ such that 
\bda
      &&  \varphi'(\hat x) +  \hat\lambda f'(\hat x)+ \hat y^*g'(\hat x)  = 0, \\ 
      &&  g(\hat x) = 0,  \\
       &&  f(\hat x) \le 0, \q \hat\lambda f(\hat x)  = 0,  \q  \hat\lambda\ge 0.   
\eda
\end{Theorem}

Let us fix a point $(\hat x,\hat\lambda,\hat y^*)$ that satisfies the system of 
relations in the above theorem (called KKT system). We shall equivalently reformulate the KKT system
as a variational inequality:
\bd
           L'_x(x,\lambda,y^*) = 0, \q g(x) = 0, \q f(x) \in N_{R^m_+}(\lambda),
\ed
where $L(x,\lambda,y^*) =  \varphi(x)  +   \lambda  f(x) + y^* g(x) $  is the Lagrangian, and 
\bd
         N_{\R^m_+}(\lambda):= \left\{ \begin{array}{cl} 
	\big\{\alpha \in \R^m  \sth  \langle \alpha, \beta -\lambda \rangle \leq 0
	{\text{ for all } \beta \in \R^m_+}\,\big\} & \text{if } \lambda \in \R^m_+,\\ \\
	\emptyset &  \text{if } \lambda \not\in \R^m_+ \end{array}\right.
\ed
is the normal cone to $\R^m_+$ at $\lambda \in \R^m$. Obviously,  the condition  
$f(\hat x) \in N_{R^m_+}(\hat \lambda)$ incorporates the conditions 
$f(\hat x) \in R^m_+$, $\hat \lambda \geq 0$, and the {\em complementary slackness} condition 
$\hat \lambda f(\hat x) = 0$. Thus the KKT optimality condition can be recast as the inclusion
$0 \in \F(x, \lambda, y^*)$, where the set-valued mapping $\F$ is defined as 
\bd
    \F(x, \lambda, y^*) := \left( \begin{array}{c}
      f(x) - N_{R^m_+}(\lambda)  \\
         g(x) \\
          L'_x(x, \lambda, y^*)
      \end{array} \right).
\ed
Our aim in this subsection is to obtain sufficient conditions for the SMsR property 
of the mapping $\F$.
For this purpose we first have to define the domain and the image spaces 
$\X$ and $\Y$ so that $\F : \X \To \Y$.  Assume that in the Banach space $(X, \| \cdot \|)$ there is another (weaker) norm 
$\| \cdot \|^\circ$, that is $\| x \|^\circ \leq\| x \|$ for every $x \in X$.
We set
\bd
    X' = \{\zeta \in X^*  \sth \zeta \mbox{ is continuous with respect to } \| \cdot \|^\circ \}
\ed
and define for any  $\zeta\in X'$ the norm
\bd
               \| \zeta \|' :=\sup_{\| x \|^\circ \leq 1}  | \zeta(x) |.
\ed
Then we define the spaces
\bd
          \X := X   \times \R^{m*}\times Y^*, \qq \Y := \R^m \times Y \times X'. 
\ed
In the space $\X$ we define following two the norms   
\bd
     \| s \|_\X := \| x \| +|\lambda|+ \| y^* \| \quad \mbox{ and } 
      \;\;  \| s \|^\circ_\X :=  \| x \|^\circ+|\lambda|+ \| y^* \|, \qquad s = (x,\lambda, y^*) \in \X.
\ed
In the space $\Y$ we define
\bd
     \| z \|_\Y =  \| z \|^\circ_\Y =  | \xi |  + \| \eta \| + \| \zeta \|' \qquad z = (\xi, \eta, \zeta) \in \Y.
\ed
Observe that due to Assumption 2  we have $f'_i(\hat x) \in X'$ and 
$ y^* g'(\hat x) \in X'$. Thus $\F: \X \To \Y$. 

We make the following ``two-norm differentiability'' assumptions for the mappings 
$\varphi$, $f$ and $g$. 

\begin{Assumption}\label{assump2}
There exists a neighborhood $\hat O$ of $\hat x$ (in the norm $\| \cdot \|$) 
such that the following 
conditions are fulfilled for all $\Delta x\in X$ such that $\hat x + \Delta x \in \hat O$:

(i) The operator $g$ is continuously Fr\'echet differentiable in $\hat O$ in the norm 
$\| \cdot \|$, and the derivative $g'(\hat x)$ is a continuous operator 
with respect to the norm $\| \cdot \|^\circ$; moreover, the following representation holds true 
\bd
       g(\hat x + \Delta x) = g(\hat x) + g'(\hat x) \Delta x + r(\Dex) \q \mbox{ with } \q
            \|  r(\Dex) \|_Y \leq \theta(\| \Dex \|) \| \Dex \|^\circ,
\ed 
where $\theta(t) \to 0$ as $t \to 0 +$. 	

(ii)There exists a bilinear mapping $Q: X \times X \to Y$ 
such that   
\bd
     g'(\hat x + \Dex)= g'(\hat x) + Q( \Dex,\cdot) + \bar r(\Dex),  \qquad
           \|Q(x_1,x_2)\|_Y \leq C \| x_1 \|^\circ \| x_2 \|^\circ,
\ed
where $C$ is a constant and $\bar r$ satisfies 
\bd
       \sup_{\| x \|^\circ \leq 1} \| \bar r(\Dex)(x) \|_Y  \leq \theta(\| \Dex \|) \| \Dex \|^\circ. 
\ed

(iii) $\varphi$ and $f$  satisfy similar conditions with bilinear mappings $Q_0$ and $Q_f$.
\end{Assumption} 

\bino
Define the quadratic functional $\Omega:X \to \R$ as
\bd
     \Omega(x)  :=  Q_0(x,x) + \hat \lambda Q_f(x, x) + \hat y^*\, Q(x,x),
\ed
and the so-called {\em critical cone} at the point $\hat x$  as
\bd
      K :=  \{ x \in X \sth \varphi'(\hat x) x\le 0, \q  f'_i(\hat x) x \le 0\; \mbox{ for }\; 
                        i \in I,\q  g'(\hat x) x = 0\}.
\ed

\begin{Assumption}\label{assump3}
There exists a constant $c_ 0 > 0$ such that 
\bd
       \Omega(x) \geq c_0\, (\| x \|^\circ)^2 \quad \forall \, x \in K.
\ed
\end{Assumption}


\bino

\begin{Assumption}\label{assump4}(strict MFCQ) 
The image  $g'(\hat x) X$ is closed and 
\bd
\ba{rcl}
   &&   \sum_{i\in I}  \lambda_i  f_i'(\hat x)+y^* g'(\hat x) = 0 
  \q \mbox{ with $\lambda_i \geq 0$ when $\hat \lambda_i = 0$} \\\\
   && \Longrightarrow \;\lambda_i = 0\q  \forall\,i \in I, \q  y^* = 0.
\ea
\ed
\end{Assumption}

\bino
The strict MFCQ (Assumptions \ref{assump4}) implies   MFCQ 
(Assumptions \ref{assump1}). 
In particular, strict MFCQ implies the condition $g'(\hat x) X=Y$. 
Moreover,  strict MFCQ  implies that 
$(\hat x,  \hat \lambda,\hat y^*)$ is the unique KKT point for the fixed $\hat x$.

\begin{Theorem}\label{th3}
Let $\hat s = (\hat x,  \hat \lambda,\hat y^*)\in\X$ be a KKT point for problem (MP), 
and let assumptions \ref{assump2}--\ref{assump4} 
be fulfilled at this point. Then the KKT optimality mapping $\F$ is 
Strongly Metrically sub-Regular at the point $(\hat s,0)$. 

More precisely, there exist constants $a > 0$, $b > 0$ and $\kappa \geq 0$ 
such that for every $y = (\xi,  \eta, \zeta )\in\Y$  such that 
$\|y\|_\Y \le b$ and for every solution $s = (x,\lambda,y^*)\in\X$ 
of the inclusion $z= (\xi, \eta, \zeta) \in \F(s)$ with $\| x - \hat x\| \leq a$ it holds that 
\bd
      \| x - \hat x \|^\circ +| \lambda  - \hat \lambda |+ \| y^* - \hat y^* \|    
   \leq \kappa ( | \xi | + \| \eta \| + \| \zeta \|').
\ed
\end{Theorem}


A proof of the above theorem is given in \cite{Osmol+Vel-21-JMAA} and 
\cite{Osmol+Vel-21-C-and-C}.

\bino
{\large\bf Additional references:}

The above result about SMsR of the optimality mapping is formulated under 
{\em strict Mangasarian-Fromovitz conditions} together with 
{\em second-order sufficient conditions}.
In the case of finite-dimensional spaces $X$ and $Y$ the result is known from 
Dontchev and Rockafellar (1998) \cite[Theorem 2.6]{AD+TR-98} and 
Cibulka, Dontchev and Kruger (2018) \cite[Section 7.1]{Cibulka+Dontchev+Kruger-18}. 
We mention that in the first of the quoted papers also local non-emptiness of $\mathcal F^{-1}$
is proved, as well as a number of related results that substantially use the finite dimensionality. 
More about regularity properties of problem ({\rm MP})
in the finite-dimensional case can be found in \cite[Chapter 8]{Klatte+Kummer-2002-book} 
and \cite[Chapter 5.2]{Bonn+Shap-book-2000}. 

The study of SMR of the KKT mapping goes back to the Robinson \cite{Robin-80}; see
\cite{Klatte+Kummer-13} for more recent results. 

Results in Hilbert spaces can be found in \cite{D+H+M+V-00} and the references therein, however,
a single norm is used in the domain space $\X$ and the SMR property is considered, therefore
the assumptions are stronger. 

Various Lipschitz stability results related to problem  ({\rm MP})
 (in Banach spaces) and the associated Lagrange multiplies are obtained in 
\cite[Chapter 4]{Bonn+Shap-book-2000}, which, as far as we can see, 
do not imply the result in the present subsection.

\section{Mayer type optimal control problem} \label{SMayer}

In this section we consider the following Mayer's type optimal control problem: 
\be\label{E1}
             \min \ph(x(0),x(1)), 
\ee 
\be\label{E2} 
          \dot x(t)=f(x(t),u(t))\q \mbox{a.e. in}\q [0,1], 
\ee
\be\label{E3} 
           G(u(t))\le0 \q \mbox{a.e. in}\q [0,1], 
\ee
where $\ph:\R^{2n}\to\R$, $f: \R^{n+m}\to\R^n$, and $G:\R^m\to\R^k$  are of class 
$C^2$, $x\in W^{1,1}$, $u\in L^\infty$.
Again, we investigate the property of {\em Strong Metric sub-Regularity}  (SMsR) 
of the {\em optimality mapping}, associated with the system of first order 
necessary optimality conditions (Pontryagin's conditions in local form) for problem 
(\ref{E1})--(\ref{E3}).

In this and the next sections we use the following standard notations. 
The euclidean norm and the scalar product in $\R^n$ (the elements of which are regarded 
as column-vectors) are denoted by $|\cdot|$ and $\lll \cdot,\cdot \rrr$, respectively. The transpose
of a matrix (or vector) $E$ is denoted by $E^\top$. For a function $\psi :\R^p \to \R^r$ of the variable 
$z$ we denote by $\psi_z(z)$ its derivative (Jacobian), represented by an $(r \times p)$-matrix.
If $r=1$, $\nabla_z \psi(z) = \psi_z(z)^\top$ denotes its gradient (a vector-column of dimension $p$). 
Also for $r=1$, $\psi_{zz}(z)$ denotes the second derivative (Hessian), represented by 
a $(p \times p)$-matrix.
For a function $\psi :\R^{p + q} \to \R$ of the variables $(z,v)$, $\psi_{zv}(z,v)$ denotes 
its mixed second
derivative, represented by a $(p \times q)$-matrix. The space $L^k([0,T],\R^r)$, with $k = 1, 2$ 
or $k = \infty$, consists of all (classes of equivalent) Lebesgue measurable $r$-dimensional
vector-functions defined on the interval $[0,T]$, for which the standard norm $\|\cdot\|_k$ is finite. 
Often the specification $([0,T],\R^r)$ will be omitted in the notations. 
As usual, $W^{1,k} = W^{1,k}([0,T],\R^r)$ denotes the space of absolutely continuous
functions $x:[0,T] \to \R^r$ for which the first derivative belongs to $L^k$.
The norm in $W^{1,k}$ is defined as $\| x \|_{1,k} := \| x \|_k  + \| \dot x \|_k$.
As before, $\B_X(x;r)$ denotes the ball of radius $r$ centered at $x$ in a metric space $X$. 

According to (\ref{E3}), the set of admissible control values is
\bd 
     U:=\{v \in \R^m \sth G(v) \le 0\}.
\ed
Let $G_i$ denote the $i$-th component of the vector $G$. For any $v\in U$ define 
the set of active indices
\bd   
        I(v)=\{i\in\{1,\ldots,k\} \sth G_i(v)=0\}.
\ed

\begin{Assumption}\label{assum4.1}{\em (regularity of the control constraints)} 
The set $U$ is nonempty and	at each point $v\in U$ the gradients $G_i'(v)$, $i\in I(v)$ are 
linearly independent.
\end{Assumption}


Define the {\em augmented Hamiltonian}
\bd
     \bar H(x,u,p,\lambda)=p\, f(x,u)+\lambda\, G(u), \q 
         p\in\R^{n*},\q \lambda\in\R^{k*}.
\ed

If the pair $(\hat x, \hat u) \in W^{1,1} \times L^\infty$ is a weak local minimizer 
in the problem \reff{E1}--\reff{E3} then there exists a unique pair $(\hat p, \hat \lambda) 
\in W^{1,1} \times L^\infty$ such that the following  
{\em optimality system} of equations and inequalities is satisfied:
\be\label{E9} 
         -\dot{\hat x}(t)+\bar H_p (\hat x(t),\hat u(t), \hat p(t), \hat \lambda) = 0 
         \q \mbox{a.e. in}\q [0,1],
\ee
\be\label{E5}
           \dot {\hat p}(t) + \bar H_x (\hat x(t),\hat u(t), \hat p(t), \hat \lambda) = 
           0 \q \mbox{a.e. in}\q [0,1],
\ee
\be\label{E4} 
            ( -\hat p(0), \hat p(1))=\varphi'(\hat x(0),\hat x(1)), 
\ee
\be\label{E6}  
          \bar H_u (\hat x(t),\hat u(t), \hat p(t), \hat \lambda) = 0,  \q \mbox{a.e. in}\q [0,1],
\ee
\be\label{E10}  
         G(\hat u(t))\le 0 \q \mbox{a.e. in}\q [0,1].
\ee
\be\label{E3a} 
        \hat \lambda(t)\ge0, \q \hat \lambda(t)  G(\hat u(t))=0\q \mbox{a.e. in}\q [0,1],
\ee
Here, the $\bar H_p,\, \bar H_x,\, \bar H_u$ denote derivatives of $\bar H$ 
with respect to the argument
indicated as subscript. Notice that here and below, the dual variables $p$ and $\lambda$ 
are treated as row vectors, while $x$, $u$, $f$, and $G$ are column vectors.

Below, the quadruple $\hat s := (\hat x, \hat u, \hat p, \hat \lambda)$ will be any 
{\em reference solution} of the optimality system \reff{E9}--\reff{E3a} 
(not necessarily resulting from an optimal pair) in the space
\bd
           \X = W^{1,1} \times L^\infty \times W^{1,1} \times L^\infty. 
\ed 
We consider this space with the following two norms: 
\be \label{EMX}
   \| s \|_\X = \| x \|_{1,1} +  \| u \|_{\infty} + \| p \|_{1,1} + \| \lambda \|_{\infty}, \qquad
   \| s \|^\circ_\X = \| x \|_{1,1} +  \| u \|_{2} + \| p \|_{1,1} + \| \lambda \|_{2}.
\ee
Moreover, we define the space
\bd
           \Y =L^1 \times L^\infty \times \R^{2n*} \times L^\infty \times L^\infty,
\ed
also endowed with two norms of its elements $y = (\xi,\pi,\nu,\rho,\eta)$:
\be \label{EMZ}
     \| y \|_\Y = \| \xi \|_{1} +  \| \pi \|_{1} + |\nu| + \| \rho \|_{\infty} + \| \eta \|_{\infty}, \qquad
     \| y \|^\circ_\Y = \| \xi \|_{1} +  \| \pi \|_{1} + |\nu| + \| \rho \|_{2} + \| \eta \|_{2}.
\ee

Observe that conditions \reff{E10}--\reff{E3a} can be equivalently reformulated 
as an inclusion
\bd
    G(\hat u(t)) \in N_{\R_+^k}(\hat \lambda(t)),
\ed
with the standard definition of the normal cone $N_{\R_+^k}$ (see Subsection \ref{SMP}).
Then one can reformulate the optimality system as a variational inequality for $(x,u,p,\lambda) \in \X$:
\bd
         \F(x,u,p,\lambda) \ni 0,
\ed
where the {\em optimality mapping} $\F : \X \To \Y$ is defined as
\bd
      \F(x, u, p, \lambda) := \left( \begin{array}{c}
       -\dot{x}+\bar H_p (x,\hat u, \hat p, \lambda) \\ 
        \dot {p} + \bar H_x (x, u, p, \lambda) \\
        ( -\hat p(0), \hat p(1)) - \ph'(\hat x(0),\hat x(1)) \\
         \bar H_u (x, u, p, \lambda) \\
            G(u) - \N_+(\lambda),
\end{array} \right).
\ed
and 
\bd
   \N_+(\lambda) := \{\eta \in L^\infty \sth \eta(t) \in N_{\R_+^k}(\lambda(t)) 
        \;\mbox{ a.e. in }[0,1]\}.
\ed
Notice that $\N_{+}(\lambda)$ is a proper subset of the normal cone at $\lambda$ to the set of all non-negative 
$L^\infty$-functions.

In order to obtain fine sufficient conditions for strong metric sub-regularity of the mapping $\F$ at 
the reference point $(\hat x, \hat u,\hat p, \hat \lambda)$ we make one more assumption.
To formulate it, we use the notations $w := (x,u)$, $\hat w := (\hat x, \hat u)$,
$q := (x(0),x(1))$, $\hat q := (\hat x(0),\hat x(1))$.
Define the usual {\em critical cone}
\bda 
       K =\Big\{\,w\in W^{1,1} \times L^\infty \!\!\!\!\!&\sth& 
                 \!\!\!\!\dot x(t)=f'(\hat w(t))w(t)\q \mbox{ for a.e. } t \in [0,1], \\ 
           &&   G'_j(\hat u(t))u(t)\le 0 \;\;\mbox{  for a.e. $t$ for which  } G_j(\hat u(t)) = 0, \\ 
     && G'_j(\hat u(t))u(t)=0 \;\;\mbox{  for a.e. $t$ for which  } \hat \lambda(t) > 0\;\; 
         j=1,\ldots,k\,\Big\}.
\eda

Next, for any $\Delta>0$  we define the extended critical cone
\bda 
       K_\Delta =\Big\{\,w\in W^{1,1} \times L^\infty \!\!\!\!\!&\sth& 
                 \!\!\!\!\dot x(t)=f'(\hat w(t))w(t)\q \mbox{ for a.e. } t \in [0,1], \\ 
           &&   G'_j(\hat u(t))u(t)\le 0 \;\;\mbox{  for a.e. $t$ for which  } G_j(\hat u(t)) = 0, \\ 
     && G'_j(\hat u(t))u(t)=0 \;\;\mbox{  for a.e. $t$ for which  } \hat \lambda(t) > \Delta\;\; 
         j=1,\ldots,k\,\Big\}.
\eda
Notice that the cones $K_\Delta$ form a non-increasing family as $\Delta\to0+$ and  
$K \st K_\Delta$ for any $\Delta > 0$.

Define the  following {\em quadratic form} of $w \in W^{1,1} \times L^\infty$:
\be\label{12}
          \Omega(w):= \langle \ph''(\hat q)q,q\rangle +
         \int_0^1 \langle  \bar H_{ww}(\hat w(t),\hat p(t), \hat\lambda(t))w(t),w(t)\rangle\dd t.
\ee

\begin{Assumption}\label{assum4.3} There exist constants $\Delta>0$ and $c_\Delta>0$ such that 
\be\label{13}
\Omega(w)\ge c_\Delta\big(|x(0)|^2 +\|u\|_2^2\big)\q  \forall\, w\in K_\Delta.
\ee
\end{Assumption}

\begin{Theorem}\label{TMayer}  
Let  Assumptions \ref{assum4.1} and \ref{assum4.3} be fulfilled.
Then the optimality mapping $\F$ is strongly metrically sub-regular (in the four metrics 
defined in \reff{EMX}, \reff{EMZ}) at the reference point 
$((\hat x, \hat u,\hat p, \hat \lambda), 0_\Y)$, where $0_\Y$ is the origin in $\Y$.
\end{Theorem}

In detail, the theorem can be formulated as follows.
There exist reals $\delta>0$
and $\kappa > 0$ such that if $z = (\xi,\pi,\nu,\rho,\eta)$ satisfies the inequality 
\be\label{21p}  
      \|\xi\|_1 + \|\pi\|_1+ |\nu|+ \|\rho\|_\infty+\|\eta\|_\infty \le \delta,
\ee 
then for any solution $s = (x,u,p,\lambda)$ of the perturbed inclusion $z \in \F(s)$ such that 
$\| x - \bar x \|_\infty + \| u - \bar u \|_\infty \le \delta$  the following estimation holds: 
\be \label{EH254}  
          \| x  - \hat x\|_{1,1} + \| u - \hat u \|_2 + \| p - \hat p\|_{1,1} + \| \lambda - \hat 
      \lambda \|_2  \le \kappa \,(\| \xi \|_{1} +  \| \pi \|_{1} + |\nu| + \| \rho \|_{2} + \| \eta \|_{2}). 
\ee
Observe that the set of points $s$ for which 
$\| x - \bar x \|_\infty + \| u - \bar u \|_\infty \le \delta$ is a ``cylindric'' neighborhood 
of $\hat s$ in the metric $\| \cdot \|_\X$. Also, the inequality in \reff{EH254}  is satisfied
in the norms $\| \cdot \|_\X^\circ$ and $\| \cdot \|_\Y^\circ$, while the neighborhoods 
of $\hat s$ and zero are with respect to $\| \cdot \|_\X$ and $\| \cdot \|_\Y$.
 
The proof of the above theorem is given in \cite{Osm+Vel_AMOP_22}.

\bino
{\large\bf Additional remarks and references.}
A similar result as Theorem \ref{TMayer} is obtained in \cite{D+H+M+V-00},
where however, the property of SMR is considered for which the conditions are stronger.

An important feature of Assumption \ref{assum4.3} is that the quadratic functional $\Omega$
has quadratic growth with respect to the norm $\| u \|_2$ of the control component of the 
elements of the extended critical cone $K_\Delta$. Such a condition is generally called {\em coercivity}
(see e.g. \cite{DH-93}). In particular, this condition is satisfied if \reff{13} holds 
for every $u \in \U - \hat u$ ($\U$ is the set of admissible controls) with the $x$-part of $w$
being the corresponding solution of the linearised equation $\dot x(t)=f'(\hat w(t))w(t)$
as in the definition of the critical cone. Assumption \ref{assum4.3} is generally weaker than 
the latter one. 
In \cite[Theorem 5]{DH-93}, Dontchev and Hager prove 
a property similar to SMR (the last term was not in use in that time) under coercivity
condition in which \reff{13} is required for all $u \in \U - \U$. This coercivity condition
proved to be sufficient for SMR and finds numerous applications in studying convergence 
of numerical methods for ODE optimal control, see e.g.
 \cite{Ang+Cor+Vel_MPS_21,CDKV-17,Dont+Hager+Veliov:RK,Dont+Hager+Veliov,%
AD+VV:CC-09,Hager-00}. 
Most of these results can, in fact, be obtained based on the SMsR property.

\section{An affine optimal control problem} \label{Saffin}

In this section, we investigate the properties of SMsR and SMR of the optimality mapping
associated with problems that are affine with respect to the control:
\be \label{Eg}
    \min \Big\{ J(u) := \int_0^T  [w(t,x(t)) + \lll d(t,x(t)),  u(t)  \rrr ] \dd t \Big\},
\ee
subject to
\be \label{Ex}
\dot x(t) = a(t,x(t))+B(t,x(t)) u(t) ,  \quad  x(0)=x^0,
\ee
\be  \label{Eu}
        u(t) \in U, \quad t \in [0,T].
\ee
The state vector $x(t)$ belongs to $\R^n$,  the control function $u$ has values 
$u(t)$ that belong to a given set $U$ in $\R^m$ for a.e.  $t \in [0, T]$.
Correspondingly, $w$ is a scalar function on $[0,T] \times \R^n$, $d$ is an $m$-
dimensional vector function, $a$ and $B$ are vector-/matrix-valued 
functions with appropriate dimensions. 
The initial state $x^0$ and the final time $T>0$ are fixed.
The set of admissible control functions $u$, denoted in the sequel by $\U$, consists of 
all Lebesgue measurable and bounded functions $u: [0,T] \to U$. 
We use the abbreviations 
\be \label{Efg}
   f(t,x,u) = a(t,x)+B(t,x) u, \qquad g(t,x,u) = w(t,x) + \lll d(t,x),  u \rrr.
\ee

Obviously the coercivity assumption in Section \ref{SMayer} is not fulfilled in general.
For problem (\ref{Eg})--(\ref{Eu}) we make the following assumption.

\begin{Assumption} \label{AA1}
The set $U$ is convex and compact; the functions $f: \R \times \R^n \times \R^m \to \R^n$
and $g: \R \times \R^n \times \R^m \to \R$ have the form as in (\ref{Efg}) and
are two times differentiable in $(t,x)$, and the derivatives
are Lipschitz continuous\footnote{\label{FNA1}The assumption of {\em global} Lipschitz 
continuity is made for convenience. Since the analysis in this paper is local 
(in a neighborhood of a reference 
trajectory $\hat x(\cdot)$ in the uniform metric), it can be replaced with 
{\em local} Lipschitz continuity. The differentiability assumption in $t$ can also be relaxed, 
especially in the consideration of the SMsR property.}. 
\end{Assumption} 

Define the Hamiltonian associated with problem (\ref{Eg})--(\ref{Eu}) as usual:
\bd
   H(t,x,p,u) := g(t,x,u) + \lll p, f(t,x,u) \rrr,\quad  p\in\R^n.
\ed
Similarly as to the previous subsection, the Pontryagin (local) maximum principle 
can be written in the form of a generalized equation (differential variational inequality)
\be \label{EFF}
    0 \in \F(x,p,u) := \left( \begin{array}{c}
	- \dot x + f(\cdot,x,u) \\
	\dot p + \nabla_{\!\! x} H(\cdot,y) \\
	\nabla_{\!u} H(\cdot,y) + N^{1,\infty}_\U(u), 
\end{array} \right),   \q y:=(x,p,u),
\ee
where 
\bd
     N^{1,\infty}_\U(u) := \{ \rho \in W^{1,\infty} \sth \lll  \rho, u' - u \rrr \leq 0 \;\;\; \forall \;
                                     u' \in \U \} 
\ed
is a subset of the normal cone to $\U$ at $u$, considered as a subspace of $L^1$. 

Assumption \ref{AA1} implies that there exists a number $M > 0$ such that for any $u \in \U$ 
the corresponding solution $x$ 
of (\ref{Ex}) and also the solution  $p$ of the adjoint equation 
$\dot p + \nabla_{\!\! x} H(\cdot,y) = 0$, $p(T) = 0$,
exist on $[0,T]$ and
\be \label{EM}
\max \{ |x(t)|, \,   |\dot x(t)|, \,  |p(t)|, \, |\dot p(t)| \} \leq M \quad \mbox{for a.e. $t \in [0,T]$}.
\ee
In what follows, $\bar M$ will be a fixed number larger that $M$.

The set-valued mapping $\F$ is considered as acting from the space 
\bd
         \X := \{(x,p,u) \in W^{1,1}\times W^{1,1}\times  L^1 : \, x(0) = x^0, \, p(T) = 0\}
         \cap \big(\B^2_{W^{1,\infty}}(0;\bar M) \times \U\big)
\ed
(where $\B^2_{W^{1,\infty}}(0;\bar M):=\B_{W^{1,\infty}}(0;\bar M)\times\B_{W^{1,\infty}}(0;\bar M)$)
with the norm
\bd
     \| (x,p,u) \|_\X := \| x \|_{1,1} +  \| p \|_{1,1} + \| u \|_{1}, 
\ed
to the space
\bd
     \Y := \{(\xi,\pi,\rho)\} = L^\infty \times L^\infty \times W^{1,\infty}.
\ed
endowed with the following two norms 
\bda
     \| (\xi,\pi,\rho) \|_\Y &:=& \|\xi\|_\infty + \|\pi\|_\infty +\|\rho\|_{1,\infty}, \\
    \| (\xi,\pi,\rho) \|_\Y^\circ &:=& \|\xi\|_1+\|\pi\|_1 +\|\rho\|_\infty.
\eda

Notice that $F(\X) \st \Y$, thanks to the independence of $\nabla_{\!u} H(\cdot,y)$ of $u$.

\bino
Let us fix a reference solution $\hat s = (\hat x, \hat p, \hat u)$ of the inclusion $0 \in \F(s)$.
We mention that such always exists since on Assumption \ref{AA1} problem  (\ref{Eg})--(\ref{Eu}) 
has a solution $(\hat x,\hat u)$. To shorten the notations  we skip arguments with ``hat" in
functions, shifting the ``hat" on the top of the notation of the
function, so that $\hat f(t) := f(t,\hat x(t),\hat u(t))$,
$\hat d(t) = d(t,\hat x(t))$, $\hat H(t) := H(t,\hat x(t),\hat u(t),\hat p(t))$,
$\hat H(t,u) := H(t,\hat x(t),u,\hat p(t))$, etc. Moreover, denote

\begin{align*}
	\hat A(t) := f_x(t,&\hat x(t),\hat u(t)), \quad
	\hat B(t) := f_u(t,\hat x(t),\hat u(t)) =B(t,\hat x(t))\\
	&\hat \sigma(t) := \nabla_u \hat H(t) = \hat B(t)^\top  \hat p(t) + \hat d(t).
\end{align*}

Let us introduce the following functional of $L^1 \ni v \mt \Gamma(v) \in \R$:
\bd
     \Gamma(v) :=\int_0^T \left[ \lll\hat H_{xx}(t) x(t), x(t) \rrr +
      2 \lll \hat H_{ux}(t) x(t), v(t) \rrr \right] \dd t,  
\ed
where $x$ is the solution of the equation $\dot{x} = \hat A x + \hat B u$ 
with initial condition $x(0) = 0$.


\begin{Assumption} \label{AA2}
There exist numbers $c_0$, $\al_0 > 0$ and $\gamma_0 > 0$ such that
\be  \label{Ecoerc1}
  \int_0^T \lll \sigma(t), v(t) \rrr  \dd t +  \Gamma(v) \geq c_0 \| v \|_1^2,
\ee 
for every $v = u - u'$ with $u, u' \in \U  \cap B_{L^1}(\hat u; \al_0)$, and
for every function $\sigma \in \B_{W^{1,\infty}}(\hat \sigma; \gamma_0)\cap (-N_\U(u'))$.
\end{Assumption}

\begin{Remark} \label{Rass_1}
{\em Notice that, in contrast to the coercivity Assumption \ref{assum4.3}, 
the above growth condition: (i) is in the space $L^1$ for the control variation; 
(ii) contains a linear term (not only quadratic).
}\end{Remark}

The following theorem is proved in \cite{Corella+Quinc+Vel-20}.

\begin{Theorem} \label{TSbi-MR}
Let Assumption \ref{AA1} be fulfilled for problem (\ref{Eg})--(\ref{Eu}) 
and let $\hat s = (\hat x, \hat p, \hat u)$ be a solution 
of the optimality system  (\ref{EFF}) for which Assumption \ref{AA2} is fulfilled. 
Let, in addition, the matrix $\hat H_{ux}(t) \hat B(t)$ be symmetric for a.e. $t \in[0,T]$.
Then the optimality mapping $F : \X \To \Y$ is strongly metrically regular around 
$(\hat s, 0)$ with respect to the above defined norm in $\X$ and norms in $\Y$.    
\end{Theorem}

The following assumption is stronger but more compact than Assumption \ref{AA2}.

\begin{Assumption} \label{AA2'}
There exist numbers $c_0$, $\al_0 > 0$ and $\gamma_0 > 0$ such that
\be \label{EA2'}
    \int_0^T |\lll \sigma(t), v(t) \rrr |  \dd t +  \Gamma(v) \geq c_0 \| v \|_1^2,
\ee 
for every function $\sigma \in \B_{W^{1,\infty}}(\hat \sigma; \gamma_0)$ 
and for every $v \in \U - \U$ with $\| v \|_1 \leq \al_0$.
\end{Assumption}

Obviously Assumption \ref{AA2'} implies \ref{AA2}, since for 
$\sigma \in - N_\U(u')$ and $u \in \U$ it holds that $\lll \sigma(t), u(t) - u'(t) \rrr \geq 0$. 

\bino
Now, we focus on the first-order term in (\ref{Ecoerc1}) under an additional condition. 

\begin{Assumption} \label{AB}
The set $U$ is a convex and compact polyhedron. Moreover, there exist numbers 
$\kappa > 0$ and $\tau > 0$ such that for every unit vector $e$ parallel to some edge 
of $U$ and for every $s \in [0,T]$
for which $\lll \hat \sigma(s), e \rrr = 0$ it holds that
\bd
     | \lll \hat \sigma(t), e \rrr | \geq \kappa | t - s |   \qquad t \in [ s - \tau, s + \tau] \cap [0,T].
\ed
\end{Assumption}

%

\begin{Proposition} \label{PB}
Let Assumptions \ref{AA1} and \ref{AB} be fulfilled. Then there exist numbers 
$c_0$, $\al_0 > 0$ and $\gamma_0 > 0$ such that
\be \label{EB}
	\int_0^T |\lll \sigma(t), v(t) \rrr |  \dd t  \geq c_0 \| v \|_1^2,
\ee 
for every function $\sigma \in \B_{W^{1,\infty}}(\hat \sigma; \gamma_0)$ 
and for every $v \in \U - \U$ with $\| v \|_1 \leq \al_0$.
\end{Proposition}

This proposition implies that the linear term in \reff{Ecoerc1} alone can ensure satisfaction 
of Assumption~\ref{AA2}. The term $\Gamma(v)/\| v \|_1^2$ may take even negative 
values, provided that they are smaller than $c_0$ in \reff{EB}.

\bino
{\bf Strong metric subregularity.}
Now we briefly discuss the SMsR property, following, slightly modifying results in \cite{Osmol+Vel-19}.

\begin{Assumption} \label{AA2p}
There exist numbers $c_0$ and $\al_0 > 0$ such that 
\reff{Ecoerc1} is fulfilled with $\sigma = \hat \sigma$ for every 
$v \in \U - \hat u$ with $\| v \|_1 \leq \al_0$.
\end{Assumption}

Notice that the last assumption is weaker than Assumption \ref{AA2} in two respects:
first, the function $\sigma$ in the latter assumption is fixed to $\hat \sigma$, and second,
the variation $v$ in Assumption \ref{AA2p} has the more specific form $v = \U - \hat u$ 
instead on $v \in \U - \U$. Of course, Remark \ref{Rass_1} also applies to Assumption~\ref{AA2}. 

Moreover, the domain space $\X$ and the range space $\Y$ of the optimality 
mapping $\F$ in \reff{EFF} change, so that $\Y$ becomes substantially larger
(that is, a larger class of disturbances is allowed): now
\bd
   \X := \{(x,p,u) \in W^{1,1}\times W^{1,1}\times  L^1 : \, x(0) = x^0, \, p(T) = 0\}
         \;\mbox{ with }   \| (x,p,u) \|_\X := \| x \|_{1,1} +  \| p \|_{1,1} + \| u \|_{1} 
\ed
and
\bd
     \Y := \{(\xi,\pi,\rho)\} = L^1 \times L^1 \times L^\infty \q\mbox{ with } \;
      \| (\xi,\pi,\rho) \|_\Y = \| (\xi,\pi,\rho) \|_\Y^\circ = \|\xi\|_1+\|\pi\|_1 +\|\rho\|_\infty.
\ed

It was proved in \cite{Osmol+Vel-19} that under Assumptions \ref{AA1} and \ref{AA2p} 
the optimality mapping $\F$ is strongly metrically subregular at the reference pair 
$(\hat x, \hat p, \hat u)$ and $0$ in the latter spaces $\X$ and $\Y$. 

An analog of Proposition \ref{PB} is also valid for the SMsR. Namely,
under Assumptions \ref{AA1} and \ref{AB}, the inequality
\bd
	\int_0^T \lll \hat \sigma(t), u - \hat u \rrr  \dd t  \geq c_0 \| u - \hat u \|_1^2 
\ed 
holds for every $u \in \U$ -- sufficiently close to $\hat u$ in $L^1$.

\bino
We mention that Assumption \ref{AA2p} is also a sufficient condition for 
strict week local optimality in $L^1$. Even more, it was proved in \cite{Osmol+Vel-19}
that the condition similar to Assumption \ref{AA2} with \reff{Ecoerc1} replaced with 
\bd  
  \int_0^T \lll \sigma(t), v(t) \rrr  \dd t +  \frac{1}{2} \Gamma(v) \geq c_0 \| v \|_1^2
\ed 
is also sufficient for strict optimality. One can prove that the latter condition is weaker than 
Assumption~\ref{AA2} in the case where $\Gamma(v)$ may take negative values.
 
\bino 
Finally, Assumption \ref{AA2p} can be weakened by replacing \reff{Ecoerc1} with 
\bd 
      \int_0^T \lll \sigma(t), u(t) - \hat u(t) \rrr  \dd t +  \Gamma(u-\hat u) 
     \geq c_0 \| u-\hat u \|_1^\kappa,
\ed
for every $u \in \U  \cap B_{L^1}(\hat u; \al_0)$, where $\kappa \in (1,2]$.
This condition implies strong H\"older subregularity of $\F$, however if $\kappa < 2$
this is true provided that the considered problem is linear-quadratic, that is,
$f(t,x,u) = A(t) x + B(t) u$, $g(t,x,u) = w(t)x + W(t) x + C(t) u$.

\bino
{\large\bf Additional references:} 
The SMR property is investigated for affine problems only in a few papers,
starting with linear systems in \cite{MQ+VV:SICON-13}, where two norms
in the image space were used. The need of these two norms for a relevant 
representation of the concept of SMR for affine problems is explained in detail in
\cite{Qui+Scar+Vel-20}. The non-linear affine case is considered in \cite{Corella+Quinc+Vel-20},
where Assumption \ref{AA2} is introduced and Theorem \ref{TSbi-MR} is proved.
A condition similar to Assumption \ref{AB} is introduced \cite{Agr+Stef+Z-02} 
in connection with a method for solving affine optimal control
problems by optimal localization of the switching points.
A similar condition is used for box-like sets $U$ in 
\cite{Felge2003} to obtain results that are related to (but stronger then) SMsR.   
The weaker Assumption \ref{AA2p} is introduced in \cite{Osmol+Vel-19},
under which SMsR of the optimality mapping is proved. 
There are numerous related results that investigate stability with respect to 
perturbations of the solutions of affine problems under various assumptions 
on the structure of the optimal control (bang-singular, bang-singular-bang, etc.), 
see e.g. \cite{Felg_Pogg_Stef-09,Pogg+Stef-20}.

There is an amount of literature on direct discretization methods for affine problems,
where the challenge is the usual discontinuity of the optimal controls. 
The property of SMsR is either explicitly used or is hidden behind the technicalities of the proofs.
Here, we mention \cite{Alt+Baier+Lem+Ger-13,Alt+F+Seyd-18,Alt+S+S-16,Osmol+Vel-19} 
for error estimates 
of the Euler scheme and \cite{Scarinci_Vel-17} for higher order estimation of a discretization technique
based on Volterra-Fliess expansion of the right-hand side of the ODE.
We also mention a series of works by Hager and co-authors that use 
the switching points as free variables and allow combinations of bang and singular arcs,
see e.g. \cite{Ag+Hager-21,Hager-23}. As noticed in \cite{Cibulka+Dontchev+Kruger-18}
(and earlier in \cite{Bonnans-94} under different conditions),
SMsR of the optimality system in principle implies convergence of the Newton method.
For affine optimal control problems this is made explicit in \cite{Felg-16,Pre_Sca_Vel_MReg-17}.

\section{Optimal control of parabolic equations} \label{SSEllip}

This section is devoted to the study of strong metric Hölder sub-regularity of the 
optimality mapping of an optimal control problem for a semilinear parabolic PDE.
Let $\Omega\subset\mathbb R^n$, $1\le n\le 3$, be a bounded domain with Lipschitz 
boundary $\partial \Omega$. Given a finite time $T>0$, we denote by 
$Q:=\Omega\times (0,T)$ the space-time cylinder, and its lateral boundary by 
$\Sigma:=\partial \Omega \times (0,T)$. We consider the optimal control problem
\begin{equation} \label{EOF}
      \  \ \min_{u \in \mathcal U}\bigg\{ J(u) := \int_Q L_0(x,t,y(x,t))+ g(x,t)u(x,t)
    \dd x\dd t\bigg\},
\end{equation}
subject to
\begin{equation} \label{befsee1}
	\left\{ \begin{array}{lll}
		\big(\frac{\partial }{\partial t}+\mathcal A\big)y+f(\cdot,y)=u\  &\text{ in }\ Q,\\
		y=0 \text{ on } \Sigma,\quad y(\cdot,0) =  y_0\ &\text{ on } \Omega.\\
	\end{array} \right.
\end{equation}
The set of feasible controls is 
\begin{equation}
	\mathcal U := \{u \in L^\infty(Q) \vert \ u_a \le u \le u_b\ \text{ for a.a. } (x,t) \in Q\},
	\label{befconstpara}
\end{equation}
where $u_a, u_b\in L^\infty(Q)$ with $ u_{a} < u_{b} $ a.e in $ Q$. 

\begin{Assumption}\label{A1}
\begin{itemize}
\item[(i)] The operator $ \mathcal A :H^1_0(\Omega)\to H^{-1}(\Omega)$, is given by
\bd
	      \A y:=-\sum_{i,j=1}^{n}\partial_{x_j}(a_{i,j}(x)\partial_{x_i} y),
\ed
where $a_{i,j}\in L^\infty(\Omega)$. Further, the  $a_{i,j}$ satisfy the uniform ellipticity 
condition
\bd
		\exists \lambda_{\mathcal A}>0:\ \lambda_{\mathcal A}\vert\xi\vert^2\leq    
                   \sum_{i,j=1}^{n} a_{i,j}(x)\xi_{i}\xi_{j} \;\;
		\textrm{ for all } \xi\in \mathbb{R}^n \;\; 	 \textrm{and a.a.}\ x\in \Omega.
\ed
		
\item [(ii)] The functions $f:Q \times \mathbb{R}\longrightarrow \mathbb{R}$
and $L(x,t,y,u):=L_0(x,t,y)+g(x,t)u:Q \times \mathbb{R}^2 \longrightarrow \mathbb{R}$ 
are Carath\'eodory function of class $C^2$ with respect to the second variable. 
In addition,
		\begin{align*}
			\left\{\begin{array}{l}
				f(\cdot,\cdot,0)\in L^\infty(Q) \text{ and }\, \frac{\partial f}{\partial y}(s,x,y)    
         \geq 0 \  \forall y \in \mathbb{R},\;\;
           L_0(\cdot,\cdot,0) \in L^1(Q), \, g\in L^\infty(Q), \vspace{1mm}\\
				\displaystyle \forall M>0\ \exists C_M, L_M, \e >0 \text{ such that }
              \forall \, |u|, |y|, |y_1|, |y_2 | \le M  \mbox{ with }  |y_2 - y_1|  \le	\e \\
       \left\vert  
      \frac{\partial f}{\partial y}(s,x,y)\right\vert+
				\left\vert\frac{\partial^2 f}{\partial y^2}(s,x,y)\right\vert \leq C_M, 
   \\ \displaystyle \Big\vert\frac{\partial L}{\partial y}(s,x,y,u)\Big\vert+ \Big\vert
      \frac{\partial^2L}{\partial y^2}(x,y,u)\Big\vert \le C_M, \\
				\left\vert\frac{\partial^2f}{\partial y^2}(s,x,y_2) - 
      \frac{\partial^2f}{\partial y^2}(s,x,y_1)\right\vert \leq L_M |y_2-y_1|, \\
  \displaystyle\left\vert\frac{\partial^2L}{\partial y^2}(s,x,y_2,u) - 
     \frac{\partial^2L}{\partial    y^2}(s,x,y_1,u)\right\vert \leq  L_M |y_2-y_1| 
         \end{array}\right.
     \end{align*}
		for almost every $(s,x) \in Q$.
	\end{itemize}
\end{Assumption}

As usual, the Hilbert space $W(0,T)$ consists of all of functions in 
$L^2(0,T\text{; }H^1_0(\Omega))$ that have a distributional derivative 
in $L^2(0,T\text{; }H^{-1}(\Omega))$, 
which is endowed with the norm
\[
		\|   y\|   _{W(0,T)}:=\|    y\|   _{L^2(0,T;H^1_0(\Omega))}+
     \| \partial y/\partial t \|   _{L^2(0,T;H^{-1}(\Omega))}.
\]
By definition, $y \in L^\infty(Q) \cap W (0, T)$ is a solution of \reff{befsee1} if
\bd
	\int_Q \Big(\frac{\dd}{\dt}y+\mathcal A y\Big)\psi \dx\dt=\int_Q (-f{(\cdot,y)} + u) 
      \psi\dx \dt, \text{ for all }\psi\in L^2(0,T, H^1_0(\Omega)).
\ed

Below we formulate as propositions some known facts that will be used in the formulation
of the main result in this sub-section. Further on, $r$ will be a fixed number satisfying
\bd
r > \max \{2, 1 + n / 2\}. 
\ed

\begin{Proposition}[\cite{ECKK2022,CM2020}]
For every $u\in  L^r(Q)$ there exists a unique solution $y_u$ of \eqref{befsee1}. 
Moreover, if $u_k\rightharpoonup u$ weakly in $L^{r}(Q)$, then 
	\begin{equation*}
		\|y_{u_k}-y_u\|_{L^{\infty}(Q)}+\|y_{u_k}-y_u\|_{L^2(0,T;H^1_0(\Omega))}\to 0.
	\end{equation*}
	\label{befestsemeq}
\end{Proposition}

This propositions implies, in particular, that a globally optimal solution of the considered
problem does exist in $L^r(Q)$, see e.g. 
\cite[Theorem 5.7]{Troltzsch2010}, hence in $L^\infty(Q)$, due to the control constraints.

We denote the control-to-state operator $G_r:L^r(Q)\to W(0,T)\cap L^\infty(Q)$, 
assigning to each $u$, the unique state $y_u$, i.e., $G_r(v):=y_u$. 

\begin{Proposition} \label{befderivative}
The control-to-state operator is of class $C^2$ and for every $u,v,w\in L^r(Q)$, 
it holds that $z_{u,v}:=\mathcal \mathcal  G_r'(u)v$ is the solution of 
\begin{align}
		\left\{\begin{array}{l}
		\big(\frac{\partial }{\partial t}+\mathcal A\big)z+f_y(x,t,y_u) z =v \;\;  {\rm in }\;\; Q,
			\\ z=0 \; \; {\rm on }\;\; \Sigma,\;\; z(\cdot,0)=0 \;\; {\rm on }\;\; \Omega,
		\end{array} \right.
		\label{befstddt}
\end{align}
and $\omega_{u,(v,w)}:=\mathcal \mathcal  G_r''(u)(v,w)$ is the solution of
\begin{align}
		\left\{\begin{array}{l}
				\big(\frac{\partial }{\partial t}+\mathcal A\big)z+f_y(x,t,y_u) z =
                  -f_{yy}(x,t,y_u) z_{u,v } z_{u,w} \;\;{\rm in }\;\; Q,
			\\ z=0 \;\;{\rm on }\;\, \Sigma,\ z(\cdot,0)=0\;\; {\rm on }\;\; \Omega.
		\end{array} \right.
		\label{befstdddt}
\end{align}
\end{Proposition}


\begin{Proposition} \label{befT3.1}
	The functional $J:L^r(Q) \longrightarrow \mathbb{R}$ is of class $C^2$. Moreover, 
	given $u, v, v_1, v_2 \in L^r(Q)$ we have
	\begin{align*}
		J'(u)v& =\int_Q\Big(\frac{dL_0}{dy}(x,t,y_u) z_{u,v}+gv \Big)\dx\dt
		= \int_Q(p_u+g)v\dx\dt,\\
		J''(u)(v_1,v_2)& = \int_Q\Big(\frac{\partial^2L_0}{\partial y^2}(x,t,y_u,u) - p_u
\frac{\partial^2f}{\partial y^2}(x,t,y_u)\Big)
		z_{u,v_1}z_{u,v_2}\dx\dt\label{befE3.5.a}.
	\end{align*}
Here, $p_u \in W(0,T) \cap C(\bar Q)$ is the unique solution of the adjoint equation
\begin{equation*}
		\left\{\begin{array}{l}\displaystyle \Big(-\frac{\dd }{\dt}+\mathcal{A}^*\Big)p + \frac{\partial f}
        {\partial y}(x,t,y_u)p=  
			\frac{\partial L}{\partial y}(x,t,y_u,u) \text{ in } Q,\\ p = 0\text{ on } \Sigma, \ 
          p(\cdot,T)=0\text{ on } \Omega.\end{array}\right.
\end{equation*}
\end{Proposition}

We introduce the Hamiltonian 
$Q\times\R\times\R\times \R \ni (x,t,y,p,u) \mapsto H(x,t,y,p,u) \in \R$, associated 
with the problem in the usual way:
\begin{align*}
	H(x,t,y,p,u):=L(x,t,y,u)+p(u-f(x,t,y)).
\end{align*}

\begin{Proposition} \label{befpontryagin}
If $\bar u$ is a weak local minimizer for problem \reff{EOF}--\eqref{befconstpara}, 
then there exist unique elements $\bar y, \bar p\in W(0,T)\cap L^\infty(Q)$ such that 
\begin{align}
   & \left\{\begin{array}{l} 	\big(\frac{\partial }{\partial t}+\mathcal A\big)\bar y + f(x,t,\bar y) = \bar u 
       \;\;{\rm in } \;\;Q,\\ \bar y =  0 \;\;{\rm on }\;\; \Sigma, \ 
           \bar y(\cdot,0)=y_0\;\;{\rm on }\;\; \Omega.\end{array}\right.
		\label{befE3.7}\\
		&\left\{\begin{array}{l}\displaystyle \Big(-\frac{\dd }{\dt}+\mathcal{A}^*\Big)p  =
      \frac {\partial H}{\partial y}(x,t,\bar y,\bar p,\bar u) \;\;{\rm in } \;\;Q,\\ \bar p = 
     0\text{ on } \Sigma, \ \bar p(\cdot,T)=0\text{ on } \Omega.
         \end{array}\right.
		\label{befE3.8}\\
		&\int_Q\frac {\partial H}{\partial u}(x,t,\bar y,\bar p,\bar u)(u - \bar u)\dx\dt \ge 0   
       \quad \forall u \in \mathcal U. \label{befE3.9}
\end{align}
\end{Proposition}

Having the optimality system \reff{befE3.7}--\reff{befE3.9}, we define in the next
paragraphs the optimality mapping corresponding to problem 
\eqref{EOF}--\eqref{befconstpara}. 
Given the initial data $y_0$ in \eqref{befsee1}, we introduce the set
\begin{equation}\label{befopmapdom}
  D(\mathcal L):=\Big\{ y\in W(0,T)\cap L^{\infty}(Q)\Big \vert \ \Big(\frac{\dd}{\dt}+\mathcal 
         A\Big)y \in L^{r}(Q), y(\cdot,0)=y_0 \Big\}.
\end{equation}
We define the operator 
$\mathcal L:=\frac{\dd}{\dt}+\mathcal A :D(\mathcal L)\to L^{r}(Q)$.
The expression $\Big(\frac{\dd}{\dt}+\mathcal A\Big)y \in L^{r}(Q)$ means that there exists 
a function $\phi\in L^r(Q)$ such that
\bd
	\int_Q \Big(\frac{\dd}{\dt}y+\mathcal A y\Big)\psi\text{ d}x\text{ d}t=\int_Q \phi 
      \psi\text{ d}x\text{ d}t, \text{ for all }\psi\in L^2(0,T, H^1_0(\Omega)).
\ed
Of course, this is satisfied with $\phi = u - f$ if $y$ solves the PDE \reff{befsee1}.
To address the adjoint equation, we define the mapping
$\mathcal L^*:D(\mathcal L^*)\to L^{r}(Q)$, $\mathcal L^*:=(-\frac{\dd}{\dt}+\mathcal A^*)$, 
where
\bd
      D(\mathcal L^*):=\Big\{ p\in W(0,T)\cap L^{\infty}(Q)\Big \vert  \Big(-\frac{\dd}{\dt}+
     \mathcal A^*\Big)p \in L^{r}(Q), p(\cdot,T)=0 \Big\}.
\ed
Using the mappings $\mathcal L$ and $\mathcal L^*$, the state equation 
\eqref{befsee1} and the adjoint equation \eqref{befE3.8} can be 
written in a short form: 
\bd
     \mathcal L y=u-f(\cdot,y), \qquad 
  \mathcal L^* p=\frac {\partial H}{\partial y}(\cdot,y_u,p,u).
\ed
We remind the definition of normal cone to the set $\mathcal{U}$ at $u \in L^1(Q)$:
\begin{equation*} 
	N_{\mathcal{U}}(u):= \left\{ \begin{array}{cl}
		\big\{\nu \in L^{\infty}(Q)\big\vert  \ \int_Q \nu (v-u)\,\text{d}x\text{d}t\le 0 \ \ \forall v 
          \in \mathcal U \big\} & \text{ if } u \in \mathcal{U}, \\
		\emptyset  & \mbox{ if } u \not\in \mathcal{U}.
	\end{array}  \right. 
\end{equation*}

Let $C$ be a fixed positive constant. Denote by $\B_C^r$ the ball of radius $C$ in the 
space $L^r(Q)$, centered at the origin. Define the metric spaces
\begin{align} \label{befEYZ}
    \X:=D(\mathcal L)\times D(\mathcal L^*)\times\U  \quad \text{and} \quad 
         \Y:= \B_C^r \times \B_C^r \times L^\infty(Q),
\end{align}
where the elements of the space $\X$ are triples $\psi :=(y,p,u)$, 
while the elements of $\Y$ are denoted by $\zeta := (\xi,\eta,\rho)$. 
The metrics in these spaces are defined as follows:
\begin{align} \label{befEnzeta}
	d_{\X}(\psi_1,\psi_2)&:=\|   y_1-y_2\|   _{L^2(Q)}+\|   p_1-p_2\|   _{L^2(Q)}+\|   
        u_1-u_2\|   _{L^1(Q)},\\
	\nonumber d_{\Y}(\zeta_1,\zeta_2)&:=\|   \xi_1-\xi_2\|   _{L^{2}(Q)}+\|  
     \eta_1-\eta_2\|   _{L^{2}(Q)}+ 	\|   \rho_1-\rho_2\|   _{L^\infty(Q)}.
\end{align}
Notice that the norms used in the space $\Y$ are weaker than the norm in which the 
ball $\B_C^r$ is taken, because $r > 2$ for $n\in \{2,3\}$.

The first order necessary optimality condition for problem 
\eqref{EOF}--\eqref{befconstpara} in Theorem \ref{befpontryagin} 
can be rewritten as the inclusion 
\be \label{EF}
   0 \in \F(y,p,u) \;\; \mbox{ with } \; 
       \F(y,p,u) :=\left( \begin{array}{c}
		\mathcal Ly + f(\cdot,y)-u \\
		\mathcal L^*p- \frac {\partial H}{\partial y}(\cdot,y,p,u) \\
		\frac {\partial H}{\partial u}(\cdot,y,p,u)+ N_{\mathcal U}(u)
	\end{array} \right).
\ee

To study SMs-R property of the optimality mapping we consider elements 
$(\xi,\eta,\rho) \in \Y$ as disturbances in the optimality system: 
$(\xi,\eta,\rho) \in \F(y,p,u)$.

Let $(\bar y, \bar p, \bar u) \in \X$ be a reference solution of the inclusion \reff{EF}.
The next assumption is crucial for obtaining the SMs-R property of the optimality 
mapping $\F$. It was proved in \cite{DJV2022b} that it is also a sufficient 
condition for weak local optimality in $L^1$ of the reference control $\bar u$. 

\begin{Assumption}[\cite{DJV2022b}]\label{befasuff2para}
There exist positive constants $c$, $\alpha$ and $\gamma\in (4/6,1]$ for $n\in \{1,2\}$ and 
$\gamma\in (5/6,1]$ for $n=3$ such that
\be \label{growthbangpara}
		J'(\bar u)(u-\bar u)+J''(\bar u)(u-\bar u)^2\geq c\|u-\bar u\|_{L^1(Q)}^{1+1/\gamma}
\ee
for all $u \in \mathcal U$ with $\|u-\bar u\|_{L^1(Q)}<\alpha$.
\end{Assumption}

Our main result is the following theorem in which the statement for the case $\gamma=1$ 
was proven in \cite{DJV2022b}. The statement for $\gamma\in (4/6,1]$ for $n\in \{1,2\}$ 
and $\gamma\in (5/6,1]$ for $n=3$ follows by straight forward adaptations.

\begin{Theorem}[\cite{DJV2022b}]\label{dududu}
Let Assumption \ref{A1} be satisfied and 
let Assumption \ref{befasuff2para} be fulfilled for the reference solution 
$\bar \psi = (y_{\bar u},p_{\bar u}, \bar u)$ of $0 \in \F(\psi)$.
Then the mapping $\F$ is strongly metrically H\"older sub-regular at 
$(\bar \psi,0)$. 
More precisely, for every $\varepsilon \in (0, 1/2]$ there exist positive 
constants $\alpha_n $ and $\kappa_n$, $n=1,2,3$, 
(with $\alpha_1$ and $\kappa_1$ independent of $\varepsilon$) such that 
for all $\psi = (y,p,u) \in \X$ with $\| u-\bar u\|_{L^1(Q)} \leq \alpha_n$ and 
$\zeta = (\xi,\eta,\rho)\in \Y$ satisfying $\zeta\in\F(\psi)$, 	
the following inequalities are satisfied.
\bda
	&\!\!\!\!\!\!\!\!\!\!\!\! \| u-\bar u \|_{L^1(Q)}\le \kappa_n \Big(\|\rho \|_{L^{\infty}(Q)}+
     \|\xi\|_{L^2(Q)}+\|\eta\| _{L^2(Q)}\Big)^{\gamma\theta_0},\\
			& \!\!\!\!\!\!\!\!\!\!\!\!   \|y^{\zeta}_u-y_{\bar u}\|_{L^2(Q)}+
             \|p^{\zeta}_u-p_{\bar u}\|_{L^2(Q)}
			\le \kappa_n \Big(\|\rho \|_{L^{\infty}(Q)}+\|\xi\|_{L^2(Q)}+
          \|\eta\| _{L^2(Q)}\Big)^{\gamma\theta},
\eda		
where 
\bda
			&  \theta_0 = \theta = 1 & \mbox{ if } \;  n = 1,  \qquad\\
			& \theta_0 = \theta = 1 - \varepsilon &  \mbox{ if } \; n= 2, \qquad \\
			&  \theta_0 = \frac{10}{11}  - \varepsilon, \;\;  \theta = \frac{9}{11}  - \varepsilon    
                  &  \mbox{ if } \; n= 3. \qquad
\eda
\end{Theorem} 

Additional results in the spirit of Theorem \ref{dududu} can be found in \cite{DJV2022b}.

\bino

{\large\bf Additional remarks and references.}
Results in the spirit of Theorem \ref{dududu} were established in \cite{DJV-22-elliptic} for affine optimal control problems subject to elliptic PDEs with Robin-type boundary conditions. Similar results were shown under the designation 
of solution stability in \cite{CDJ2023} considering several distinct assumptions on the joint growth of the first 
and second variation of the objective functional which are, in the elliptic case, weaker than Assumption \ref{befasuff2para}.\newline
In \cite{CJV2024Ser} the Strong Metric Hölder Subregularity of the optimality mapping 
was investigated for an optimal control problem governed by a semilinear parabolic PDE as in the discussion above, but with the additional constraint that the controls have a fixed state-distribution. This allowed the authors, due to the stronger a-priori estimates for the involved PDEs, 
to improve the estimations of Theorem \ref{dududu} substantially.\newline
The approach introduced in \cite{DJV2022b} relies on $L^s-L^1$ type estimates for the linearized PDEs allowing the estimations that are key in the proof of \cite[Lemma 11]{DJV2022b}. Motivated by this result, the notion of Strong Metric Hölder Subregularity was further investigated for optimal control problems constrained by the $2$-dimensional Navier--Stokes equation in \cite{DJNS202} and for the Boussinesq system in \cite{JS2024}. Regarding Assumption \ref{befasuff2para}, it is important to notice, that it can be satisfied only by bang-bang optimal controls. Still in \cite{DJNS202}, it was shown that for stability under perturbations in the normal cone, and thus SMHsR, it is necessary to have a growth which is, depending on the sign of the second variation, either close to Assumption \ref{befasuff2para} or agrees with it. This highlights the generality of Assumption \ref{befasuff2para} in the bang-bang situation, which is also supported by the results in \cite{DW} which state that the occurrence of bang-bang optimal controls is generic in affine optimal control problems. Subsequently, in the case of a linear quadratic control problem that possesses a nonnegative second variation, it was shown in \cite{DL}, that for those problems, the notion of SMsR is not only equivalent to a growth of the objective functional as in Assumption \ref{befasuff2para} and discussed in \cite{DJNS202}, but also to a Polyak-Lojasiewicz-type inequality. The implication of Assumption \reff{befasuff2para} for the investigation of error estimates for the numerical solution was evidenced in \cite{J2023} where it was shown that for an optimal control problem subject to a semilinear elliptic PDE, Assumption \ref{befasuff2para} is sufficient for error estimates for 
the numerical approximation generated by a Finite Element Method (FEM) with piecewise 
constant controls, and that the notion of Strong Metric Hölder Subregularity can be applied for 
the investigation of error estimates for the numerical approximation generated by a FEM with 
a variational discretization.


\end{document}